\documentclass[12pt]{amsart}
\usepackage{amsmath,amssymb,amsthm,amsfonts,enumerate,array,color,lscape,fancyhdr,layout,pst-all}
\usepackage{xcolor}
\usepackage[english]{babel}
\usepackage{young}
\usepackage{float}
\usepackage[all]{xy}
\usepackage[normalem]{ulem}
\usepackage[active]{srcltx}
\usepackage{mathrsfs}
\usepackage{etoolbox}
\usepackage[margin=1in]{geometry}

\usepackage{xspace}
\newlength\yStones
\newlength\xStones
\newlength\xxStones

\makeatletter
\def\Stones{\pst@object{Stones}}
\def\Stones@i#1{%
  \pst@killglue%
  \begingroup%
  \use@par%
  \setlength\xxStones{\xStones}%
  \expandafter\Stones@ii#1,,\@nil
  \endgroup
  \global\addtolength\xStones{0.6cm}%
  \global\addtolength\yStones{-7.5mm}}%
\def\Stones@ii#1,#2,#3\@nil{%
  \rput(\xxStones,\yStones){%
    \psframebox[framesep=0]{%
      \parbox[c][6mm][c]{11mm}{\makebox[11mm]{$#1$}}}}%
  \addtolength\xxStones{1.2cm}%
  \ifx\relax#2\relax\else\Stones@ii#2,#3\@nil\fi}
\makeatother

\def\Stone#1{\fbox{\makebox[12mm]{\strut#1}}\kern2pt}

\newcommand{\bk}{\Bbbk}
\newcommand{\A}{\mathbb{A}}

\newcommand{\bP}{\mathbb{P}}

\newcommand{\fp}{\mathfrak{p}}

\newcommand{\gl}{\mathfrak{gl}}
\newcommand{\Ann}{\mathrm{Ann}}
\newcommand{\Supp}{\mathrm{Supp}}

\newcommand{\cO}{\mathcal{O}}
\newcommand{\cD}{\mathcal{D}}

\newcommand{\cM}{\mathcal{M}}

\newcommand{\Frac}{\mathrm{Frac}}
\newcommand{\rank}{\mathrm{rank}}

\newcommand{\Spec}{\mathrm{Spec}}
\newcommand{\Tor}{\mathrm{Tor}}

\newcommand{\End}{\mathrm{End}}
\newcommand{\Der}{\mathrm{Der}} 
\newcommand{\cV}{\mathcal{V}}

\newtheorem{theorem}{Theorem}[section]
\newtheorem{lemma}[theorem]{Lemma}
\newtheorem{corollary}[theorem]{Corollary}
\newtheorem{proposition}[theorem]{Proposition}

\theoremstyle{definition}
\newtheorem{example}[theorem]{Example}

\newtheorem{definition}[theorem]{Definition}

\allowdisplaybreaks

\newtoggle{details}
\iftoggle{details}{
\newcommand{\details}[1]{{\color{blue}\noindent\textbf{Details:}{#1}}}
                  }{
                  \newcommand{\details}[1]{}
                  }
                
\begin{document}
\begin{title}[Annihilators of $A\mathcal{V}$-modules and differential operators]{Annihilators of $A\mathcal{V}$-modules and differential operators}
\end{title}

\author[E. Bouaziz]{Emile Bouaziz}

\author[H. Rocha]{Henrique Rocha}


\begin{abstract}
For a smooth algebraic variety $X$, we study the category of finitely generated modules over the ring of function of $X$ that has a compatible action of the Lie algebra $\cV$ of polynomials vector fields on $X$. We show that the associated representation of $\cV$ is given by a differential operator of order depending on the rank of the module. The order of the differential operator provides a natural measure of the complexity of the representation, with the simplest case being that of $D$-modules.
\end{abstract}




\maketitle




\section*{Introduction}
The representation theory of infinite dimensional Lie algebras is a fairly difficult problem to approach in any generality. Some very particular infinite dimensional Lie algebras have Cartan subalgebras and root decomposition which allows us to consider classes of modules similar to the ones considered for finite-dimensional simple Lie algebras. But, for instance, only in the last decade a breakthrough in the theory of weight modules over Witt algebras and Cartan type Lie algebras of type $W$ was accomplished \cite{BF16, GS22} even though these algebras were already known by the pioneers the Lie theory \cite{Lie80,Car09}. On the other hand, a typical Lie algebra does not have those structures and common techniques of Lie theory may not apply. This is the case of Lie algebras of polynomial vector fields on an affine smooth curve of positive genus which is a simple Lie algebra that does not have neither nilpotent elements nor semisimple elements \cite{BF18}.

To obtain a manageable theory one must substantially restrict the class of representation under consideration. In this paper, we will consider representations of the Lie algebra of vector fields on a smooth algebraic variety that have a compatible action of the algebra of functions. Let $X$ be a smooth algebraic variety, $A = \bk[X]$ its coordinate ring and $\cV = \Der(A)$ the Lie algebra of polynomial vector fields on $X$. In the paper \cite{BFN19}, the authors introduce \emph{$A\cV$-modules}, namely an $A$-module $M$ with a compatible action of $\cV$ in the sense that the \emph{Leibniz formula} holds, $\eta ( f m) = \eta(f) m + f(\eta m)$ for all $f \in A$, $\eta \in \cV$, $m \in M$.

This restriction to this class of representation is natural from the point of view of algebraic geometry and reflects the original incarnation of the Lie algebras in question as derivations of the ring of functions on the variety. For example, $D$-modules furnish such representations. Importantly, these are of a special type and a general $A\cV$-module need not come from a $D$-module. For example, the natural action of vector fields on differential forms gives an $A\cV$-module which is not a $D$-module. The associated representation $\cV \rightarrow \End_{\bk}(M)$ does not need to be $A$-linear.

The main goal of this paper is to prove that $A\cV$-modules that are finitely generated as $A$-modules in a sense satisfy a weaker form of $A$-linearity, see Theorem \ref{theorem:repisdifop}, or Theorem \ref{theorem:geomformulation} for a geometric formulation. Namely, the representation $\cV \rightarrow \End_{\bk}(M)$ (which we consider as a left $A$-module) is a differential operator of some order that depends on the rank of $M$. Note that this implies that the action of $\cV$ on $M$ sheafifies, which we stress \emph{is not} a formal consequence of the definitions. In fact, we prove directly that the action sheafifies and then use this in the course of proving that it is a differential operator.

The content of this paper is as follow. In Section \ref{section:001preliminars} we collect some definitions and prove that every finite $A\cV$-module is a projective $A$-module. In Section \ref{section:002theliealgebraAV} we prove necessary identities of some elements of a certain Lie algebra related to $A$ and $\cV$. Using these identities, we get certain annihilators of $A\cV$-modules in Section \ref{section:003annihilators}. In Section \ref{section:004localization} we show that the representation of $\cV$ in a finite $A\cV$-module sheafifies. Finally, in Section \ref{section:005globaltheory} we define an algebraic geometric notion of $A\cV$-modules and we show the associated representation $\cV \rightarrow \End_{\bk}(M)$ is a differential operator.

\section{Preliminaries}\label{section:001preliminars}

Let $\bk$ be an algebraically closed field of characteristic $0$. Throughout the paper the ground field is $\bk$. All vector spaces, linear maps, algebras and tensor products are assumed to be over $\bk$ unless otherwise stated. 

Let $X \subset \A^n$ an irreducible smooth algebraic variety of dimension $d$. Denote $A= \bk[X]$ the coordinate ring of $X$ and $\cV= \Der(A)$ the Lie algebra of polynomial vector fields on $X$. We remark that $A$ is a Noetherian integral domain, and $\cV$ is a simple Lie algebra \cite{Jor86,Sie96}.

An $A\cV$-module $M$ is a module over both $A$ and $\cV$ such that 
\[
\eta\cdot(f \cdot m) = \eta(f)\cdot m +f \cdot ( \eta \cdot m) \quad \text {for each} \ \eta \in \cV, \ f \in \A, \ m \in M.
\]
Equivalently, $M$ is a module over the smash product $A \# U(\cV)$, where $U(\cV)$ denotes the universal enveloping algebra of $\cV$. We say that an $A\cV$-module $M$ is finite if $M$ is finitely generated as an $A$-module.

The smash product $A \# U(\cV)$ is an associative algebra, thus the commutator defines a Lie algebra structure on it. For all $f,g \in A$, $\eta, \mu \in \cV$, we have that
\[
(f \# \eta)(g\# \mu) = f\eta(g) \# \mu + fg \# \eta \mu.
\]
thus
 \[
 [f \# \eta, g \# \mu] = fg \# [\eta, \mu] + f  \eta(g) \# \mu - g  \mu(f) \# \eta.
 \]
It follows that $A \# \cV$ is a Lie subalgebra of $A \# U(\cV)$. If $M$ is an $A\cV$-module, the set of all operators of $\gl_{\bk}(M)$ given by the action of $A \# \cV$ corresponds to the $A$-module generated by the image of the representation $\cV \rightarrow \gl_{\bk}(M)$. 


For an $A$-module $M$ and a proper prime ideal $\fp \in \Spec(A)$, we define \[\rank_{\fp} (M) = \dim_{\bk (\fp)} ( \bk (\fp) \otimes M), \] where $\bk (\fp)$ is the residual field of the local algebra $A_{\fp}$. The number $\rank_{\fp} (M)$ is the minimal number of generators of $M_{\fp}$ as an $A_{\fp}$-module. We define $\rank(M) = \rank_{(0)} (M )= \dim_{\Frac(A)} \Frac(A) \otimes_A M $. By Nakayama's Lemma, $\rank (M) \leq \rank_{\fp} (M)$ for each $\fp \in \Spec(A)$. Furthermore, $M$ is a projective $A$-module if and only if $\rank(M) = \rank_{\fp}(M)$ for each $\fp \in \Spec(A)$ \cite[Exercise 20.13]{Eis95}.

\begin{lemma}\label{lemma:annihilatoriszero}
Let $M$ be an $A\cV$-module. Then 
\[
\Ann_A(M) = \{ f \in A \mid \forall m \in M \ fm = 0 \} = 0.
\]
\end{lemma}
\begin{proof}
    The ideal $\Ann_A(M)$ is a proper $A\cV$-submodule of $A$, therefore $\Ann_A(M)$ is trivial.
\end{proof}
For an $A$-module $M$, we denote by $\Lambda_A^r(M)$ the $r$-exterior power of $M$ over $A$. If $M$ is an $A\cV$-module, then $\Lambda_A^r(M)$ is also an $A\cV$-module where the action of $\cV$ is given by the coproduct of $U(\cV)$. This will be useful to prove the following proposition.
\begin{proposition}\label{proposition:avmodisprojective}
    Let $M$ be a finite $A\cV$-module, then $M$ is a projective $A$-module. In particular,
    \[
    \Tor_A(M) = \{m \in M \mid \exists f \in A, \ f \neq 0, \  \text{such that} \ fm =0   \} =0
    \]
\end{proposition}
\begin{proof}
    Set $r= \rank(M)$, so $\Lambda^r_A(M)$ is nontrivial. We will prove that $\Lambda_A^{p}(M) =0 $ for all $p > r$. Suppose $\Lambda^{p}_A(M) \neq 0$ for some $p>r$. Since localization commutes with exterior powers and $\Lambda_{\Frac(A)}^{p}(\Frac(A) \otimes_A M)=0$, we have that there exists a nonzero proper prime ideal $\fp$ of $A$ such that $\left ( \Lambda_A^p(M) \right)_{\fp} = \Lambda_{A_{\fp}}^p(M_{\fp}) =0$.
    Therefore, \[\fp \in \Supp (\Lambda^{p}_A(M) = V(\Ann_A(\Lambda^{p}_A(M ))) \neq 0,\]
    thus $\Ann_A(\Lambda^{p}_A(M) )\neq 0$. However, $\Lambda^{p}_A(M)$ is an $A\cV$-module and the annihilator of any $A\cV$-module is zero by Lemma \ref{lemma:annihilatoriszero}. Hence, $\Lambda^{p}_A (M)$ must be trivial. Therefore, $\Lambda_A^{p}(M) =0 $ for all $p >r$, which implies that $\rank_{\fp}(M) = r $ for every $\fp \in \Spec(A)$. We conclude that $M$ is projective.
\end{proof}

\section{The Lie algebra $A \# \cV$}\label{section:002theliealgebraAV}

 All identities proved in this section do not rely on $X$. We only need a commutative algebra $A$ and $\cV = \Der(A)$.

\begin{definition}
For each $f \in A$, $\eta \in \cV$ and $p \geq 1$ consider the following element of $A \#\cV$
\begin{equation*} \label{equation:definitionomega}
    \Omega_p(f,\eta ) =  \sum_{k=0}^p (-1)^{k} \binom{p}{k} f^{p-k} \# f^k \eta \in A \# U(\cV).
\end{equation*}
\end{definition}
\begin{example} $\Omega_{1}(f,\eta)$ vanishes identically if and only if our module $M$ is a $\cD$-module. In general the elements $\Omega_{p}$ measure a certain \emph{higher} non-linearity in the $\cV$ action on $M$. As a further example, the reader can check that in the case of the natural adjoint action of $\cV$ on itself, the elements $\Omega_{2}$ all vanish.\end{example}

We will prove identities of those elements inside the Lie algebra $A \# \cV$. The first interesting property of these elements is that they commute with $A \cong A \# \bk \subset A \# U(\cV)$. Therefore, their actions on an $A\cV$-module can be seen as elements of $\End_A(M)\subset\End_{\bk}(M)$.

\begin{lemma}\label{lemma:omegaactioncommuteswitha}
For all $f,g \in A$, $\eta \in \cV$, and $p\geq 1$,
\[
\Omega_p(f,\eta) (g \# 1 ) = (g\# 1 ) \Omega_p(f,\eta).
\]
\end{lemma}
\begin{proof}
Using that $\displaystyle{\sum_{k=0}^p (-1)^{k} \binom{p}{k}f^r }=0$, we see that
\begin{align*}
    & \Omega_p(f,\eta) (g \# 1 ) - (g\# 1 ) \Omega_p(f,\eta) \\
    = & \sum_{k=0}^p (-1)^k \binom{p}{k} \left ( (f^{p-k} \# f^k \eta) (g \# 1 ) - (g \# 1 )  (f^{p-k} \# f^k \eta)   \right)\\
    = & \sum_{k=0}^p (-1)^k \binom{p}{k} \left ( (f^{p} \eta (g) \# 1 +f^{n-k}g \# \eta - g f^{p-k} \# f^k \eta   \right)  =0 
\end{align*}
\end{proof}

The following lemma will be used frequently when we calculate other brackets inside $A \# \cV$. The proofs of the various identities satisfied by the elements $\Omega_{p}$ follow by standard manipulations with Lie brackets, but we include some details anyway, for the convenience of the reader.
\begin{lemma}\label{lemma:commutatorofomegas}
For every $f \in A$, $\eta, \mu \in \cV$, and $p,q \geq 1$,
\[
\left[\Omega_p(f,\eta), \Omega_q(f, \mu) \right] = \Omega_{p+q}(f,[\eta,\mu] ) +p \Omega_{p+q-1} (f,\mu(f)\eta) -q \Omega_{p+q-1} (f,\eta(f)\mu)
\]
\end{lemma}
\begin{proof}
Using that $\displaystyle{\sum_{k=0}^r (-1)^{k} \binom{r}{k}h^r =0}$ for all $r$ and $h \in A$,
\begin{align*}
    & \left[\Omega_p(f,\eta), \Omega_q(f, \mu) \right] \\
    = & \sum_{k=0}^p \sum_{l=0}^q (-1)^{k+l} \binom{p}{k} \binom{q}{l} \left ( f^p \eta(f^{q-l}) \# f^{q-l}\mu -f^{q} \mu(f^{p-k}) \# f^{k} \eta + f^{p-k} f^{q-l}\# [f^k \eta, f^l \mu] \right ) \\
    = & \sum_{k=0}^p \sum_{l=0}^q (-1)^{k+l} \binom{p}{k} \binom{q}{l} f^{p+q-k-l}\# [f^k \eta, f^l \mu] \\
    = & \sum_{k=0}^p \sum_{l=0}^q (-1)^{k+l} \binom{p}{k} \binom{q}{l} f^{p+q-k-l}\# \left ( l f^{k+l-1}\eta(f)\mu - k f^{k+l-1} \mu(f)\eta + f^{k+l} [\eta,\mu] \right) \\
\end{align*}

Let's separate it in three sums. Set $u=k+l$. Thus, the coefficient at $f^{p+q-u} \# f^{u} [\eta,\mu]$ is the same as in $y^u$ in
\begin{align*}
    & \sum_{k=0}^p \sum_{l=0}^q (-1)^{k+l} \binom{p}{k} \binom{q}{l} y^{k+l} \\
    = & \sum_{k=0}^p (-1)^k \binom{p}{k} y^k \left ( \sum_{l=0}^q (-1)^{k+l} \binom{q}{l} y^{l} \right ) \\
    = & (1-y)^p (1-y)^q = (1-y)^{p+q} \\
    = & \sum_{u=0}^{p+q} (-1)^u \binom{p+q}{u} y^u
\end{align*}
Hence, 
\[
\sum_{k=0}^p \sum_{l=0}^q (-1)^{k+l} \binom{p}{k} \binom{q}{l} f^{p+q-k-l} \# f^{k+l}[\eta,\mu] = \Omega_{p+q}(f,[\eta,\mu]).
\]

The coefficient at $f^{p+q-u} \# f^{u-1} \eta(f) \mu$ is the same as in $y^u$ in
\begin{align*}
    &\sum_{k=0}^p \sum_{l=0}^q (-1)^{k+l} \binom{p}{k} \binom{q}{l} l y^{k+l} \\
    = & \sum_{l=0}^q (-1)^l \binom{q}{l} l y^l \left (\sum_{k=0}^p (-1)^k \binom{p}{k} y^k \right ) \\
    = & \left ( y \frac{d}{d y} \right  )\left (\sum_{l=0}^q (-1)^l \binom{q}{l} y^l  \right  ) \left (\sum_{k=0}^p (-1)^k \binom{p}{k} y^k \right )  \\
    = &   \left (\left ( y \frac{d}{d y} \right  ) (1-y)^q \right  ) (1-y)^p = -qy(1-y)^{p+q-1} \\
    = & -q \sum_{u=0}^{p+q-1} (-1)^{u} \binom{p+q-1}{u} y^{u+1}.
\end{align*}
Hence,
\begin{align*}
    & \sum_{k=0}^p \sum_{l=0}^q (-1)^{k+l} \binom{p}{k} \binom{q}{l} l f^{p+q-k-l} \# f^{k+l-1} \eta(f) \mu \\
    & = -q \sum_{u=0}^{p+q-1} (-1)^{u} \binom{p+q-1}{u} f^{p+q-u-1} \# f^{u} \eta(f) \mu  \\
    & = -q\Omega_{p+q-1}(f,\eta(f)\mu).
\end{align*}
Similarly, 
\begin{align*}
 & - \sum_{k=0}^p \sum_{l=0}^q (-1)^{k+l} \binom{p}{k} \binom{q}{l} k f^{p+q-k-l} \# f^{k+l-1} \eta(f) \mu \\
    = &  p \sum_{u=0}^{p+q-1} (-1)^{u} \binom{p+q-1}{u} f^{p+q-u-1} \# f^{u} \mu(f) \eta  \\
    = &  p \Omega_{p+q-1} ( f, \mu(f) \eta) 
\end{align*}

We conclude that 
\[
\left[\Omega_p(f,\eta), \Omega_q(f, \mu) \right] = \Omega_{p+q}(f,[\eta,\mu] ) +p \Omega_{p+q-1} (f,\mu(f)\eta) -q \Omega_{p+q-1} (f,\eta(f)\mu)
\]
\end{proof}

The following identities bear a resemblance to those of \cite[Section 2]{BF18}.

\begin{lemma}\label{lemma:bracketrelations001}
For every $g,h \in A$, $\eta, \mu \in \cV$, $p,q \geq 1$,

\begin{enumerate}
    \item\label{itemlemma01:identitieslemma02} $[\Omega_p (f,\eta), \Omega_q (f,g\mu) ] - [\Omega_p (f,g\eta), \Omega_q (f,\mu) ] = \Omega_{p+q} (f, \eta(g) \mu + \mu(g) \eta).$
    \item\label{itemlemma02:identitieslemma02} $[\Omega_p (f,g\eta), \Omega_q (f,h\eta) ] - [\Omega_p (f,\eta), \Omega_q (f,gh\eta) ] = 2\Omega_{p+q}(f,h\eta(g) \eta)$
    \item\label{itemlemma03:identitieslemma02} $[\Omega_p (f,\eta), \Omega_q (f,g\eta) ] - [\Omega_p (f,g\eta), \Omega_q (f,\eta) ] = 2 \Omega_{p+q} (f,  \eta(g) \eta )$;
    \item\label{itemlemma04:identitieslemma02} $[\Omega_p (f,g\eta), \Omega_q (f,\eta(h)\eta) ] - [\Omega_p (f,\eta), \Omega_q (f,g\eta(h)\eta) ] = 2\Omega_{p+q}(f,\eta(g)\eta(h) \eta)$;
    \item\label{itemlemma05:identitieslemma02} $\Omega_{p+q} (f, g \eta(\eta(h)) \eta ) =\Omega_{p+q} (f, \eta(g \mu(h))\eta  )  - \Omega_{p+q} (f,\eta(g) \eta(h)\eta ) $
\end{enumerate}
\end{lemma}
\begin{proof}
For item \eqref{itemlemma01:identitieslemma02}, we have that $[\eta,g \mu] - [g\eta,\mu] = \eta(g) \mu + \mu(g)\eta $. Therefore,
\begin{align*}
 & [\Omega_p (f,\eta), \Omega_q (f,g\mu) ] - [\Omega_p (f,\eta), \Omega_q (f,g\mu) ]\\
= & \Omega_{p+q}(f,[\eta,g\mu] ) +p \Omega_{p+q-1} (f,g\mu(f)\eta) -q \Omega_{p+q-1} (f,\eta(f)g\mu)\\  
&- \Omega_{p+q}(f,[g\eta,\mu] ) - p \Omega_{p+q-1} (f,\mu(f)g\eta) +q \Omega_{p+q-1} (f,g\eta(f)\mu) \\ 
= & \Omega_{p+q}(f,[\eta,g\mu] - [g\eta, \mu] ) = \Omega_{p+q}(f, \eta(g)\mu+  \mu(g) \eta)
\end{align*}

Now we proceed to item \eqref{itemlemma02:identitieslemma02}. We have that $[\eta,gh\eta] - [g\eta, h \eta ] = 2 h\eta(g)\eta$. Hence
\begin{align*}
& [\Omega_p (f,\eta), \Omega_q (f,gh\eta) ] - [\Omega_p (f,g\eta), \Omega_q (f,h\eta) ] \\
= & \Omega_{p+q}(f,[\eta,gh\eta] ) + p\Omega_{p+q-1}(f,gh\eta(f) \eta) -q \Omega_{p+q-1}(f,\eta(f)g h\eta) \\
& - \Omega_{p+q}(f,[g\eta,h\eta] ) - p \Omega_{p+q-1} (f,h\eta(f)g\eta) +q \Omega_{p+q-1}(f,g\eta(f) h \eta) \\
= & \Omega_{p+q}(f, [\eta,gh\eta] - [g\eta, h \eta ] ) = 2 \Omega_{p+q}(f, h\eta(g)\eta)
\end{align*}  

All other items follows from \eqref{itemlemma01:identitieslemma02} and \eqref{itemlemma02:identitieslemma02}.
\end{proof}

\begin{lemma}\label{lemma:bracketofpsiandderivation}
For all $f \in A$, $\eta, \mu \in \cV$, and $p \geq 1$, 
\[
[\Omega_{p}(f,\eta),  1\#  \mu ]  = \Omega_p(f,[\eta,\mu]) +p \Omega_{p-1}(f,\mu(f)\eta) - p\mu(f)  \Omega_{p-1}(f,\eta)
\]
\end{lemma}

\begin{proof} This follows from arguments very similar to those given in Lemma \ref{lemma:commutatorofomegas}.
\end{proof}

\section{Annihilators of finite $A\cV$-modules}\label{section:003annihilators}

\begin{definition}
For an $A\cV$-module $M$, we define the \emph{annihilator} $\Ann(M)$ by
\[
\Ann(M) = \left \{ x \in A \# U(\cV) \mid xm = 0 \quad \text{for each} \ m \in M \right \}.
\]
\end{definition}

Note that $\Ann(M)$ is a left ideal of $A \# U(\cV)$, while $\Ann_A(M)$ is an ideal of $A$.

\begin{example} In the case of $M=\Omega^{1}_{A},$ we have that for all $f,\eta$, $\Omega_{2}(f,\eta)\in \Ann(M)$.\end{example}

We wish to eventually generalize the above phenomenon. More specifically we will prove that for all $f \in A$, there exists $N>0$ such that $\Omega_p(f,\eta) \in \Ann(M)$ for each $p>N$ and $\eta \in \cV$, \emph{given that} $M$ is of finite type. We will use the identities proven for the elements $\Omega_{p}$ above in order to do so.

\begin{lemma}\label{lemma:lemma001avmodules}
Let $M$ be a finite $A\cV$-module with $r = \rank_{A}(M)$, $f \in A$ and $\eta \in \cV$.  For any $p>r^2$, there exists $a_1,\dots,a_{r^2} \in A$ and $b\in A \setminus \{0\}$ such that $b\Omega_p(f,\eta) + \sum_{i=1}^{r^2} a_i \Omega_{i}(f,\eta) \in \mathrm{Ann} (M)$.
\end{lemma}
\begin{proof}
Since $M$ is finitely generated with rank $r$, we have that $\End_A(M)$ is a finitely generated $A$-module with rank at most $r^2$. By Lemma \ref{lemma:omegaactioncommuteswitha}, the action of $\Omega_i(f,\eta)$ commutes with the action of $A$ in $M$ for each $i\in \{1,2,\dots,r^2, p\} $, we see that $\{\Omega_k(f,\eta)\mid k =1,\dots,r^2 \} \cup \{\Omega_p(f,\eta) \} $ defines a family of endomorphisms in $\End_A(M)$. Therefore, it must be $A$-linearly dependent. Thus, there exists $a_1,\dots,a_N \in A$ and $b \in A$, $b \neq 0$, such that $b\Omega_p(f,\eta) + \sum_{i=1}^{r^2} a_i \Omega_{i}(f,\eta) \in \mathrm{Ann} (M)$.
\end{proof}

\begin{lemma}\label{lemma:argument1}
Let $M$ be a finite $A\cV$-module with rank $r$, $f \in A$ and $\eta \in \cV$ such that $\eta(f) \neq 0$. Then exists $ N \geq r^2$ that depends on $r$ such that $\Omega_p(f,\eta(f)^{r^2}\eta) \in \mathrm{Ann} (M)$ for all $p \geq N$.
\end{lemma}

\begin{proof}
Let $m>r^2$, then by Lemma \ref{lemma:lemma001avmodules} there exists $a_1,\dots,a_{r^2} \in A$ non all zero and $a_{r^2+1} \in A \setminus \{0\}$ such that $\sum_{i=1}^{r^2+1} a_i \Omega_{m_i}(f,\eta) \in \mathrm{Ann} (M)$ where $m_i=i$ if $i \leq r ^2$ and $m_i = p$ if $i=r^2+1$. Thus, for every $m \in M$,
\begin{align*}
    0 & = \Omega_{m_1}(f,\eta)\left (\sum_{i=1}^{r^2+1} a_i \Omega_{m_i}(f,\eta) \right )m \\ 
    & =  \left (\sum_{i=1}^{r^2+1} a_i \Omega_{m_i}(f,\eta) \right ) \Omega_{m_1}(f,\eta) m +\left (\sum_{i=1}^{r^2 +1} a_i \left [ \Omega_{m_1}(f,\eta), \Omega_{m_i}(f,\eta) \right ]\right )m \\
    & = \left (\sum_{i=2}^{r^2+1} a_i (m_1 - m_i) \Omega_{m_1+m_i-1} (f,\eta(f)\eta) \right ) m.
\end{align*}
Therefore, $\sum_{i=2}^{r^2+1} a_i (m_1 - m_i) \Omega_{m_1+m_i-1} (f,\eta(f)\eta) \in \mathrm{Ann}(M)$. Now, for all $m \in M$
\begin{align*}
   0 & = \Omega_{m_1+m_2-1}(f,\eta(f)\eta) \left (\sum_{i=2}^{r^2+1} a_i (m_1 - m_i) \Omega_{m_1+m_i-1} (f,\eta(f)\eta) \right ) m \\
   & = \left (\sum_{i=2}^{r^2+1 } a_i(m_1 - m_i) [\Omega_{m_1+m_2-1}(f,\eta(f)\eta) , \Omega_{m_1+m_i-1} (f,\eta(f)\eta) \right ) m \\
   & = \left (\sum_{i=3}^{r^2+1} a_i (m_1 - m_i)(m_2-m_i) \Omega_{2m_1+m_2+m_i-2}(f,\eta(f)^2\eta) \right ) m .
\end{align*}
We may do this process $r^2$ times to conclude that $b\Omega_{p+N}(f,\eta(f)^{r^2}\eta)  \in \mathrm{Ann}(M)$ for some $0 \neq b \in A $ and $N \geq 0$ that depends on $r$. Since $\mathrm{Tor}_A(M) = 0$, $\Omega_{m+N}(f,\eta(f)^{r^2}\eta)  \in \mathrm{Ann}(M)$. This holds for every $p > r^2$, $\Omega_{k+r^2+N}(f,\eta(f)^{r^2}\eta)  \in \mathrm{Ann}(M)$ for every $k\geq 1$.
\end{proof}

\begin{corollary}\label{corollary:thereexistsderivation}
Let $M$ be an $A\cV$-module. For all $f \in A$ that exists $\mu \in \cV$ with $\mu(f)\neq 0$, there exists $\eta \in \cV$ with $\eta(f)\neq 0$ such that $\Omega_p(f,\eta) \in \mathrm{Ann}(M)$ for every $p>N$.
\end{corollary}

\begin{lemma}\label{lemma:itemsonannihilator001}
Let $M$ be an $A\cV$-module, $f\in A$, and $\eta \in \cV$. Suppose that for each $p >N$ $\Omega_p(f,\eta) \in \mathrm{Ann}(M)$ for some $N>0$. Then for all $g,h \in A$ 
\begin{enumerate}
    \item $\Omega_{2N+1+k} (f,  \eta(g) \eta ) \in \mathrm{Ann}(M)$ for all $k \geq 1$;
    \item $\Omega_{3N+2+k}(f,\eta(g)\eta(h) \eta) \in \mathrm{Ann}(M)$ for all $k \geq 1$;
    \item $\Omega_{3N+2+k} (f, g \eta(\eta(h)) \eta ) \in \mathrm{Ann}(M)$ for all $k \geq 1$.
\end{enumerate}
\end{lemma}
\begin{proof}
It follows from items \eqref{itemlemma03:identitieslemma02}, \eqref{itemlemma04:identitieslemma02} and \eqref{itemlemma05:identitieslemma02}  of Lemma \ref{lemma:bracketrelations001}.
\end{proof}

\begin{proposition}\label{proposition:idealannm}
Let $M$ be an $A\cV$-module, and $f \in A$. Let  $\eta \in \cV$ with $\eta(f)\neq 0$ and $N >0$ such that $\Omega_p(f,\eta) \in \mathrm{Ann}(M)$ for every $p>N$. Let $g,h \in A$ and $I_{g,h,f,\eta}$ be the principal ideal of $A$ generated by $\eta(g)\eta(\eta(h))$. Then for every $q \in I_{g,h,f,\eta}$, $\tau \in \cV$ and $p > 3N+4$, we have that $\Omega_{p}(f,q\tau) \in \mathrm{Ann}(M)$.
\end{proposition}

\begin{proof}
The ideal $I_{g,h,f,\eta}$ is generated by elements with the form $q = p \eta(g)\eta(\eta(h))$ with $p \in A$. By Lemma \ref{lemma:bracketrelations001}  and Lemma \ref{lemma:itemsonannihilator001},
\begin{align*}
    & [\Omega_{3N+2+k}(f,p \eta(\eta(h)) \eta ), \Omega_l ( f, g \tau)  ] - [\Omega_{3N+2+k}(f,gp\eta(\eta(h)) \eta), \Omega_l(f, \tau) ] \\ & - \Omega_{3N+2+k+l} (f, \tau(g) p \eta(\eta(h)) \eta )\\
    = & \Omega_{3N+2+k+l}(f,\tau(g) p \eta(\eta(h)) \eta +p \eta(g) \eta(\eta(h)) \tau) - \Omega_{3N+2+k+l} (f, \tau(g) p \eta(\eta(h)) \\
    = & \Omega_{3N+2+k+l}(f, q \tau) \in \mathrm{Ann}(M)
\end{align*}
for each $k,l \geq 1$, and $\tau \in \cV$. Therefore, $\Omega_{3N+4+k}(f, q \tau) \in \Ann(M)$ for every $k \geq 1$, $q \in I_{g,h,f,\eta}$, $\tau \in \cV$.
\end{proof}

\begin{definition}
For an ideal $I$ of $A$, define $I^{(0)} = I$ and $I^{(k)} $ as the ideal of $A$ generated $\{ g,\mu(g) \mid g \in I^{(k-1)}, \ \mu \in \cV \}$.
\end{definition}

\begin{lemma}\label{lemma:lemma02applicationstoavmodules}
Let $M$ be an $A\cV$-module, and $f \in A$. Suppose that $I$ is an ideal of $A$ such that $\Omega_p(f,q\tau) \in \Ann(M)$ for every $p>N$ for some $N> 0$. Then for each $p > N+k$, $\Omega_p(f,g \tau) \in \Ann(M)$ for all $g \in I^{(k)}$ and $\tau \in \cV$.
\end{lemma}
\begin{proof}
By Lemma \ref{lemma:bracketofpsiandderivation},
\begin{align*}
0 &= [\Omega_{p+1}(f,g\tau),1 \# \mu ]v \\
& = \Omega_{p+1}(f,[g\tau,\mu])v +(p+1)\Omega_{p} (f,\mu(f)g \tau)v - (p+1)\mu(f)\Omega_p(f,g\tau)v\\
& = - \Omega_{p+1}(f,\mu(g) \tau )v + \Omega_{p+1}(f,g [\tau,\mu] )v \\
& =  - \Omega_{p+1}(f,\mu(g) \tau )v
\end{align*}
for every $g \in I$, $\mu ,\tau \in \cV$, $p > N$ and $v\in M$. Thus, $\Omega_{p+1}(f,\mu(g) \tau ) \in \Ann(M)$ for every $g \in I$ and $\mu \in \cV$.

Furthermore, for every $g \in I$ and $h \in A$, we have that $gh \in I$ and 
\[
 \Omega_p(f,h\mu(g)\tau) = \Omega_p(f,\mu(gh)\tau) -\Omega_p(f,g\mu(h)\tau) \in \Ann(M).
\]
Hence, for every $g \in I^{(1)}$, we have that $\Omega_p(f,g \tau ) \in \Ann(M)$ for every $p > N+1$. Since $I^{(k)}=(I^{(k-1)})^{(1)}$, inductively, $\Omega_p(f,g \tau) \in \Ann(M)$ for every $p>N+k$, $\tau \in \cV$ and $g \in I^{(k)}$.
\end{proof}

\begin{lemma}\label{lemma:fpspantangentplane}
For every $f \in A $ with $f \notin \bk$ and $p \in X$, there exists  $\mu_1,\dots,\mu_l \in \cV$ for some $l\geq 1$ such that $(\mu_1 \circ \mu_2 \circ \dots \circ \mu_l )(f)(p) \neq 0$.
\end{lemma}
\begin{proof}
The proof is similar to \cite[Proposition 3.3]{BF18} and we will use its notation. Let $p \in X$, $t_1,\dots,t_s$ be local parameters centered at $p$, $\tau_1,\dots,\tau_s \in \cV$ with $\tau_i(t_j) = h \delta_{ij}$ for some $h \in A$ with $h(p)\neq 0$. Write 
\[
\sum_{\alpha} f_{\alpha} t^{\alpha} \in \bk[[t_1,\dots,t_s]]
\]
the Taylor series at $p$ of $f$. Choose $\beta=(\beta_1,\dots,\beta_s)$ such that $f_{\beta} \neq 0$ and $| \beta |$ is minimal in $\{|\alpha| \mid f_{\alpha} \neq 0 \}$. If $|\beta| = 0$, then $f(p) = f_{\beta} \neq 0 $. If $|\beta|>0$, then 
\[
\left ( \frac{\partial }{\partial t_i} \right )^{\beta_1} \cdots \left ( \frac{\partial }{\partial t_i} \right )^{\beta_s} f_{\beta} 
\]
is a nonzero multiple of $1$. Therefore, the Taylor series at $p$ of
\[
\tau_1^{\beta_1} \circ \cdots \circ \tau_s^{\beta_s} (f),
\]
has a nonzero constant term. Therefore, $\tau_1^{\beta_1} \circ \cdots \circ \tau_s^{\beta_s}(f)(p) \neq 0$.
\end{proof}

\begin{corollary}\label{lemma:therexistsderivation}
For all $f \in A$, $f \notin \bk$, there exists $\eta \in \cV$ with $\eta(f) \neq 0$. 
\end{corollary}
\begin{proof}
 By Lemma \ref{lemma:fpspantangentplane}, there exists  $\mu_1,\dots,\mu_l \in \cV$ for some $l\geq 1$ such that $(\mu_1 \circ \mu_2 \circ \dots \circ \mu_l )(f)(p) \neq 0$. In particular, $\mu_l(f) \neq 0$.
\end{proof}

\begin{theorem}\label{proposition:psifinann}
Let $M$ be a finite $A\cV$-module, and $f \in A$. Then there exists $N_f$, that depends on $f$, such that $\Omega_p(f,\eta) \in \Ann(M)$ for each $p > N_f$, and $\eta \in \cV$.
\end{theorem}
\begin{proof}
If $f \in \bk$, then $\Omega_p (f,\eta) = 0$ for all $p \geq 1$. Suppose $f \notin \bk$. By Corollary \ref{lemma:therexistsderivation}, there exists $\mu \in \cV$ such that $\mu(f) \neq 0$. By Corollary \ref{corollary:thereexistsderivation}, there exists $\eta \in \cV$ with $\eta(f)\neq 0$ and $N >0$ such that $\Omega_p(f,\eta) \in \mathrm{Ann}(M)$ for every $p>N$. Since $\eta(f) \neq 0$, there exists $g \in \{f,f^2 \}$ such that $\eta(f)\eta(\eta(g)) \neq 0$. By Proposition \ref{proposition:idealannm}, there exists a nonzero ideal $I$ of $A$ such that $\Omega_p(f,q \tau) \in \Ann(M)$ for every $q \in I$, $\tau \in \cV$ and $p > 3N+4$. Since $A$ is Noetherian and
\[
I \subset I^{(1)} \subset I^{(2)} \subset \cdots
\]
is an ascending chain of ideals of $A$, we have that $I^{(k)} = I^{(l)} $ for every $l \geq k$ for some $k\geq 1$. Let $p \in X$. If there exists $g \in I$ such that $g(p) \neq 0$, then we are done. Otherwise, by Lemma \ref{lemma:fpspantangentplane} there exists $g \in I^{(l)}$ for some $l$ such that $g(p) \neq 0 $. Since $I^{(l)} \subset I^{(k)}$ or $I^{(l)} = I^{(k)}$, we have that $g \in I^{(k)}$. Therefore, for every $p \in X$, there exists $g \in I^{(k)}$ such that $g(p) \neq 0$. By Hilbert's Nullstellensatz, $I^{(k)}= A$. By Lemma \ref{lemma:lemma02applicationstoavmodules}, for every $g \in I^{(k)}=A$ and $p>3N+4+k$, $\Omega_p(f,g \eta) \in \Ann(M)$ for every $\eta \in \cV$. In particular, $\Omega_p(f,\eta) \in \Ann(M)$ for each $p>N_f$ where $N_f = 3N+4+k$.
\end{proof}

\section{Localizing $A\cV$-modules}\label{section:004localization}
Let $M$ be an $A\cV$-module, and $f \in A$, $f \neq 0$. Then
\[
M_f = A_f \otimes_A M
\]
    is an $A_f$-module, where $A_f = \left \{\frac{a}{f^l} \mid a \in A, \ l \geq 0  \right \}$ is the localization of $A$ by the multiplicative set $\{f^k \mid k \geq 0\}$. For example, $\cV_f = \Der(A)_f \cong \Der(A_f)$. The open set $D(f)  = \{ p \in X \mid f(p) \neq 0\} \subset X$  is an irreducible smooth affine variety, and $A_f = \bk[D(f)] $. Therefore, we may consider the question whether $M_f$ is an $A_f\cV_f$-module.

We wish to define an action of $\cV_f$ in such a way that $M_f$ is a module over $A_f \# U(\cV_f)$. If $\eta \in \cV \subset \cV_f$, then its action on $M_f$ must be defined as
\[
( 1 \# \eta  )\left (\frac{m}{f} \right ) =-\frac{\eta(f) m}{f^2} + \frac{1}{f} \left(\eta m \right )
\]
for each $m \in M_f$. 

Denote $\Omega_0(f,\eta) = 1 \# \eta$. By Proposition \ref{proposition:psifinann}, there exists there exists $N_f$ such that $\Omega_p(f,\eta) \in \Ann(M)$ for every $p > N_f$. Hence, the sum
\begin{equation*}
    \sum_{p=0}^{\infty} \frac{1}{f^{p+1}} \Omega_{p}(f,\eta)m 
\end{equation*}
while infinite, is \emph{locally finite} in the sense that it converges to an operator for all $\eta \in \cV$, and $f \in A$. Define 
\begin{equation}\label{equation:definitionactionsheaf}
    \left (1 \# \frac{\eta}{f} \right )m = \sum_{p=0}^{\infty} \frac{1}{f^{p+1}} \Omega_{p}(f,\eta)m = \sum_{p=0}^{N_f} \frac{1}{f^{p+1}} \Omega_{p}(f,\eta)m
\end{equation}
for all $\eta \in \cV$, $m \in M_f$. We need to show the action defined does not depend of the choice of representation of $\eta \in \cV_f$. It is sufficient to prove the action of $\frac{f\eta}{f}$ and $\eta$ coincide for every $\eta \in \cV_f$. In order to do it, we will need the following lemma.
\begin{lemma}\label{lemma:relationforwelldefined}
    For each $f\in A$, $p \geq 0$, and $\eta \in \cV_f$,
    \[
    \Omega_p(f,f\eta) = f\Omega_p(f,\eta ) - \Omega_{p+1} (f,\eta).
    \]
\end{lemma}
\begin{proof}
This follows from the recurrence  $\binom{p+1}{k} - \binom{p}{k} = \binom{p}{k-1}$ and definition of $\Omega_p$.
\end{proof}
By the previous lemma,
\begin{align*}
    \left (1 \# \frac{f\eta}{f}  \right ) m & = \sum_{p=0}^{N_f} \frac{1}{f^{p+1}} \Omega_{p}(f,f\eta)m \\
    = & \sum_{p=0}^{N_f} \frac{1}{f^{p+1}} \left (f\Omega_p(f,\eta ) - \Omega_{p+1} (f,\eta) \right ) m\\
    = &  \left ( 1 \# \eta + \sum_{p=1 }^{N_f } \frac{1}{f^{p}}\Omega_p(f,\eta)  - \sum_{p=0}^{N_f-1} \frac{1}{f^{p+1}} \Omega_{p+1} (f,\eta) \right ) m\\
    = &\left ( 1 \# \eta \right  ) m
\end{align*}
for each $m \in M_f$. Therefore, the action \eqref{equation:definitionactionsheaf} is well-defined, so it remains to prove it satisfies the Leibniz rule and it defines a representation of $\cV_f$. 
\begin{lemma}
   For all $\eta \in \cV_f$, $g\in A_f$, $m \in M$,
    \[
    \sum_{p=0}^{N_f} \left (\frac{1}{f^{p+1}} \# 1 \right )\Omega_p(f,\eta) (g \# 1) = \frac{\eta(g)}{f} \# 1 + (g \# 1) \sum_{p=0}^{N_f} \left ( \frac{1}{f^{p+1}}\# 1\right )\Omega_p(f,\eta)
    \]
\end{lemma}
\begin{proof}
    It follows from Proposition \ref{lemma:omegaactioncommuteswitha}.  
\end{proof}
 Hence,
    \[
    \left( 1 \# \mu \right ) \left (  ( g \# 1)  m \right ) = \left( (1 \# \mu ) ( g \# 1) \right )  m   =  \left ( \mu(g) \# 1 \right )  m + ( g \# 1 )\left (  (1 \#  \mu) m \right)
    \]
    for every $m \in M_f$, $\mu \in \cV_f$, $g \in A$.

It remains to prove that \eqref{equation:definitionactionsheaf} defines a representation of $\cV_f$. In order to prove it, we will need the following lemma.

\begin{lemma}\label{lemma:actionofetaoverfsquarecubic}
For every $\eta \in \cV_f$, and $m \in M_f$,
\begin{enumerate}
    \item $\displaystyle{\left ( 1 \# \frac{\eta}{f^2} \right )m= \sum_{k=0}^{N_f} \frac{k+1}{f^{k+2}} \Omega_k(f,\eta) m}$,
    \item $\displaystyle{\left (1 \# \frac{\eta}{f^3} \right ) m = \sum_{k=0}^{N_f} \frac{(k+1)(k+2)/2}{f^{k+3}} \Omega_k(f,\eta) m}$.
\end{enumerate}
\end{lemma}
\begin{proof}
 It follows from the definition of $\Omega_p$ and similar arguments to those given in Lemma \ref{lemma:argument1}. 
\end{proof}

\begin{lemma}
    $M_f$ is a $\cV_f$-module with the action given by \eqref{equation:definitionactionsheaf}.
\end{lemma}
\begin{proof}
Fix $m \in M_f$. Since \eqref{equation:definitionactionsheaf} is well-defined for all elements of $V_f$, we only need to prove that 
\[
\left [\frac{\eta}{f}, \frac{\mu}{f} \right ] m = - \frac{\eta (f)}{f^{3}}  \mu m+ \frac{\mu(f) }{f^{3} } \eta m+ \frac{[\eta,\mu]}{f^{2}} m = \left [\sum_{p=0}^{N_f} \frac{1}{f^{p+1}} \Omega_{p}(f,\eta), \sum_{p=0}^{N_f} \frac{1}{f^{p+1}} \Omega_{p}(f,\mu) \right ] m.
\]

By Lemma \ref{lemma:commutatorofomegas} and Lemma \ref{lemma:actionofetaoverfsquarecubic}, 
    \begin{align*}
    & \left [\sum_{p=0}^{N_f} \frac{1}{f^{p+1}} \Omega_{p}(f,\eta), \sum_{p=0}^{N_f} \frac{1}{f^{p+1}} \Omega_{p}(f,\mu) \right ] m\\
    = & \sum_{k=0}^{N_f}\sum_{l=0}^{N_f} \frac{1}{f^{k+l+2}} \left [ \Omega_{k}(f,\eta),   \Omega_{p}(f,\mu) \right ] m\\
    =&  \sum_{k=0}^{N_f} \sum_{l=0}^{N_f} \frac{1}{f^{k+l+2}} \left ( k \Omega_{k+l-1}(f,\mu(f)\eta ) -l \Omega_{k+l-1}(f,\eta(f)\mu ) + \Omega_{k+l}(f,[\eta,\mu ] ) \right ) m\\
    = & \sum_{u=0}^{2N_f} \left ( \frac{(u+1)(u+2) /2}{f^{u+3}} \Omega_{u}(f,\mu(f)\eta ) -\frac{(u+1)(u+2)/2}{f^{u+3}} \Omega_{u}(f,\eta(f)\mu ) + \frac{u+1}{f^{u+2}} \Omega_u(f, [\eta,\mu]) \right ) m \\
    = &  \frac{\mu(f) }{f^{3} } \eta m- \frac{\eta (f)}{f^{3}}  \mu m+ \frac{[\eta,\mu]}{f^{2}} m
\end{align*}

\end{proof}

Since the the number of generators of $M_f$ as an $A_f$-module is less or equal than the number of generators of $M$ as an $A$-module, we have that $M_f$ is a finitely generated as an $A_f$-module. 

\begin{theorem}\label{theorem:avmodshefifies}
If $M$ is a finite $A\cV$-module and $f \in A$, $f \neq 0$, then $M_f =A_f \otimes_A M$ is a finite $A_f\cV_f$-module, where the action of $A_f$ is given by left side multiplication and
\[
\left ( \frac{\eta}{f^k}\right ) m = \sum_{p=0}^{\infty} \frac{1}{f^{k(p+1)}} \Omega_p(f^k,\eta)m
\]
for each $\eta\in \cV$.
\end{theorem}

\begin{proof} This follows from a combination of the above calculations. \end{proof}

\begin{corollary}\label{corollary:avmodshefifiesrestriction}
    Let $M$ be a finite $A\cV$-module and $f,g\in A$ be non zero elements. If $\eta \in \cV_f$ and $\mu \in \cV_g$ are such that $\eta = \mu \in \cV_{fg} $, then $\eta m = \mu m $ for all $m\in M$.
\end{corollary} 
\begin{proof}
    It follows from direct calculation using the action defined on the previous theorem.
\end{proof}

\begin{proposition}\label{proposition:omegaarezero}
    Let $M$ be a finite $A\cV$-module, then there exists $N$ that depends on the rank of $M$ such that $\Omega_p(f,\tau) \in \Ann(M)$ for every $f \in A$, $\tau \in \cV$, and $p>\rank(M)^2$.
\end{proposition}
\begin{proof}
    Let $f \in A$.  If $f \in \bk$, then $\Omega_1(f,\eta) = 0$ for every $\eta \in \cV$. Assume $f \notin \bk$, then there exists $\mu \in \cV$ such that $\mu(f)\neq 0$. Set $\eta = \frac{\mu}{\mu(f)} \in \cV_{\mu(f)}$, then $\eta(f)=1$. By Lemma \ref{lemma:commutatorofomegas}, 
    \[
    [\Omega_1(f,\eta),\Omega_p(f,\eta ) = (1-p) \Omega_p(f,\eta).
    \]
    Consider $F=\Frac(A)$ the field of fractions of $A$, and $\overline{M} = F \otimes_A M$, then we may see each $\Omega_p(f,\eta)$ as an element of the vector space $\End_F(\overline{M})$ of $F$-linear endomorphisms of $\overline{M}$. We have that $\Omega_1(f,\eta)$ acts on $\End_F(\overline{M})$ by commutation $\Omega_1(f,\eta) \cdot A=  [\Omega_1(f,\eta),A]$ for each $A \in \End_F(\overline{M})$. The elements of the set $\left \{ \Omega_p(f,\eta) \mid p =1,\dots,\rank(M)^2+1 \right  \} \subset \End_F(\overline{M})$ are eigenvectors of $\Omega_1(f,\eta)$ with distinct eigenvalues. Therefore, there exists $p \in 1,\dots,\rank(M)^2+1$ such that $\Omega_p(f,\eta)$ is trivial. However, by Lemma \ref{lemma:commutatorofomegas}
    \[
        [\Omega_a(f,\eta), \Omega_b(f,\eta)] = (a-b) \Omega_{a+b-1} (f,\eta)
    \]
    thus $\Omega_p(f,\eta)=0$ implies $\Omega_q(f,\eta) = 0 $ for every $q > p$. In particular, for all $q > \rank(M)^2$ we have $\Omega_{q}(f,\eta)  \in \Ann(\overline{M})$. By Proposition \ref{proposition:idealannm}, $\eta(f)\eta(\eta(f)) =1$ implies that $\Omega_q(f,\tau) \in \Ann(\overline{M})$ for every $\tau \in \Der_{\bk} (F)$, $q> \rank(M)^2$. Since $\cV$ injects itself in $\Der_{\bk}(F)$ and $M$ is a torsion-free $A$-module, we have that $\Omega_p(f,\tau) \in \Ann(M) $ for every $p > \rank(M)^2$, and $\tau \in \cV$.
\end{proof}

The Lie algebra $A\# \cV$ is a $A\otimes A$-module with action given by $(a\otimes b) (f\# \eta) = af \# b\eta$. For each $f_1, \cdots, f_p\in A$ and $\eta \in A$, define 
\[
\Omega((f_1,\cdots, f_p), \eta) = \prod_{i=1}^p (f_i\otimes 1-1\otimes f_i) (1\# \eta). 
\]
In particular, if $f_i=f_j$ for each $i=1,\dots, p$, then $\Omega((f_1,\dots, f_p), \eta)=\Omega_p(f_1,\eta)$.
\begin{corollary}\label{corollary:omegaisdifferenrial}
    For each $p>\rank(M)^2$, $\Omega ((f_1,\dots ,f_p), \eta)\in \Ann(M) $. 
\end{corollary}
\begin{proof}
    Using that each $\Omega_p(f_i, \eta) \in \Ann(M) $ and 
    \[
\Omega_p\left (  \sum_{i=1}^p a_i f_i, \eta  \right) = \sum_{l_1+\cdots +l_p=p}\binom{p}{l_1,\cdots, l_p} a_1^{l_1}\cdots a_p^{l_p}\prod_{i=1}^p (f_i\otimes 1-1\otimes f_i)^{l_i}(1\# \eta) \in \Ann(M)
    \] 
    for each $a_1,\cdots, a_p\in \bk$, a linear algebra argument shows that $\Omega((f_{i_1}, \cdots, f_{i_p}), \eta) \in \Ann(M) $ for every $1\leq i_1, \dots, i_p\leq p$. 
\end{proof}

\section{Global theory of $A\cV$-module}\label{section:005globaltheory}
Let $Y$ be a scheme with structure sheaf $\mathcal{O}$. Consider the sheaf $\Theta$ over $Y$ given by $\Gamma(U,\Theta) = \Der(B)$ for each open affine set $U=\Spec(B) \subset Y$. A sheaf $\cM$ over $Y$ is called \emph{infinitesimally equivariant}, or \emph{infeq} for short, if for each affine open set $U= \Spec(B) \subset Y$ we have that $\Gamma(U,\cM)$ is a $\Gamma(U,\Theta)$-module such that
\[
\eta\cdot(f \cdot m) = \eta(f)\cdot m +f \cdot ( \eta \cdot m) \quad \text {for each} \ \eta \in \Gamma(U,\Theta), \ f \in B, \ m \in \Gamma(U,\cM).
\]
If $\cM$ is a vector bundle, we call $\cM$ an \emph{infeq bundle}. 

\begin{example}
    The tangent sheaf $\Theta$ over $Y$ is an infeq sheaf, as well as every $D$-module over $Y$. However, as we have stressed above, most infeq modules do not come from $D$-modules, indeed this is already the case for the adjoint action of $\Theta$ on itself. The sheaf of $\mathcal{J}^n\cO$ of $n$-order jets of sections is another example of an infeq sheaf. Indeed, one expects every sheaf of local differential geometric origin (we are being purposefully imprecise here) to admit the structure of an infeq sheaf.
\end{example} 

\begin{example} In the case of the projective line $\mathbb{P}^{1}$, every line bundle is an infeq sheaf in a unique manner. Indeed, $\Omega\cong \mathcal{O}(-1)$ is infeq in the usual manner. It has a unique square root, $\mathcal{O}(-1)$, which therefore must also be infeq. The tensor powers and duals of $\mathcal{O}(-1)$ are all automatically infeq modules, and so we are done. With some work one can show that on an elliptic curve, the only infeq line bundles are $D$-modules, and further that on a curve of genus $g\geq 2$, every line bundle can be given the structure of an infeq bundle, now in multiple ways. The proof of these statements is standard algebro-geometric obstruction theory and we will not go into details in this note. 
   
\end{example}

\details{

\begin{example} 
Take $Y=\bP^n$ the $n$-dimensional projective space, then infeq bundles may be created by gluing infeq bundles on $\A^n$. Assume $n=1$. For $\lambda \in \bk$ consider the one-dimensional $\bk\cong\gl(1)$-module $\bk v_{\lambda}$ given by $1v_{\lambda}=\lambda v_{\lambda}$. Then $T_{\lambda}(D(h)) =\bk [t]_h \otimes \bk v_{\lambda} $ is a $\bk[t]_h \# U\left(\bk[t]_h\frac{\partial}{\partial t}\right) $-module with 
\[
g(f \otimes v_{\lambda}) = (gf)\otimes v_{\lambda}, \quad g\partial (f \otimes v_{\lambda}) = g\partial(f)\otimes v_{\lambda} + \partial(g) f \otimes \lambda v_{\lambda}
\]
for each $g, f \in \bk [t]_h$. Therefore, $T_{\lambda} $ is a infeq bundle on $\A^1$ with rank $1$.
Construct $\bP^1$ by gluing $\A^1_1 =\Spec (\bk[x]) $ and $\A^1_2 =\Spec (\bk[y]) $ along the isomorphism $(\A_1^1)_x \cong (\A_2^1)_y$ given by $x^{-1} \leftrightarrow y$. Consider the infeq sheaf $T_{\lambda}$ on $\A^1_1$ and $T_{\lambda'}$ on $\A^1_2$. Then it is possible to glue them to an infeq sheaf on $\bP^1$ if and only if $\lambda'=-\lambda$. 
\end{example}
}

We may reformulate the definition of infeq sheaves using a different approach. First let us recall the \emph{Atiyah algebra} associated to a sheaf, $\cM$, on $Y$. For a detailed discussion of these, and related constructions, the reader is referred to \cite{BS88}. We have that $\mathrm{Diff}^{\leq 1} (\cM, \cM)$, the sheaf of differential operators from $\cM$ to itself of order at most one. This is equipped with a \emph{symbol}, $\sigma$, to $\Theta \otimes \End_{\cO} (\cM) $. The preimage under $\sigma $ of the sheaf $\Theta$ is called \emph{Atiyah algebra} of $\cM$, and denoted $\mathrm{At}(\cM) $. This is naturally a Lie algebroid with anchor given by the symbol, and is in fact the Lie algebroid of infinitesimal symmetries of the pair $(Y, \cM) $, cf. \cite{BS88}. If $\cM$ is an infeq sheaf on $Y$, then the image of the associated representation $\Theta \rightarrow \End(\cM) $ lies in $\mathrm{At}(\cM)$ and this map is a splitting of the symbol $\sigma$. On the other hand, a sheaf $\cM$ on $Y$ is an infeq sheaf if it is equipped with a choice of $\bk$-linear Lie algebra splitting of the symbol $\sigma : \mathrm{At} (\cM) \rightarrow \Theta $. We refer to the splitting as the \emph{Lie map}, and denote it $L$. This gives us a different definition of infeq sheaves, and it allows us to define infeq sheaves in different geometric contexts, notably smooth or complex analytic such. A geometric reformulation of the main theorem (Theorem \ref{theorem:repisdifop}) of this note can be given as follows,

\begin{theorem}\label{theorem:geomformulation} Let $\mathcal{F}$ be a finite type infeq sheaf on a smooth variety $X$. Then the associated Lie map, $L$, is a differential operator of order bounded above by $\operatorname{rank}(\mathcal{F})^2$. \end{theorem}
\begin{proof}
    It is a reformulation of the Theorem \ref{theorem:repisdifop}, which will be proved below.
\end{proof}

\subsection{Global $A\cV$-module associated to $A\cV$-module}

If $Y=\Spec(B)$ is an affine scheme, it is not clear that the associated sheaf $\tilde{M}$ of a module $M$ over $B\# U(\Der(B))$ is an infeq sheaf. However, the answer is positive when we take a smooth irreducible algebraic variety and a finitely generated module as we showed on previous sections. Now, if we start with a coherent infeq sheaf $\cM$ over the smooth irreducible algebraic variety $X$, then global sections of $\cM$ defines an $A\cV$-module by definition, and it is finitely generated since it $\cM$ is coherent. To summarize, we have an equivalence of categories.

\begin{proposition}\label{proposition:infeqsheafiseqvtoavmod}
    The category of finite $A\cV$-modules over a smooth $\bk$-algebra $B$ is equivalent to category of infeq bundles over the smooth algebraic variety $X=\operatorname{Spec}(B)$.
\end{proposition}
\begin{proof} 
If $X=\Spec(A)$ is a smooth algebraic variety and $\cM$ is an infeq bundle, then by definition $\Gamma(X,\cM)$ is a finitely generated $A \# U(\cV)$-module for $\cV=\Der(A)$.

Let $M$ be an $A\cV$-module.
    Consider the quasi-coherent sheaf $\tilde{M}$ on $X$ given by $\Gamma(D(f),\tilde{M}) = \tilde{M}(D(f))=M_f = A_f \otimes_A M$ for each element $f \in A$, where $D(f) = \{ \fp \in \Spec(A) \mid f \notin \fp \}$. Then $\tilde{M}$ is coherent, because $M$ is finitely generated, and it is a vector bundle, since $M$ is a projective $A$-module by Proposition \ref{proposition:avmodisprojective}. By Theorem \ref{theorem:avmodshefifies} and Corollary \ref{corollary:avmodshefifiesrestriction}, $\Gamma(D(f),\tilde{M})$ is a $\Gamma(D(f),\tilde{\cV})$-module, its action satisfy the Leibniz Rule, it is compatible with restrictions, and it agrees on intersections. Therefore, $\tilde{M}$ is an infeq bundle on $X$.
\end{proof}

As we commented before, the associated representation $\rho: \cV \rightarrow \End(M)$ is a choice of $\bk$-linear splitting of the symbol $\mathrm{At}(M) \rightarrow \cV$. A-prior it need not be linear over the algebra of functions, but we can ask instead for a weakened form of linearity. More precisely, we can ask if $\rho$ is a differential operator. As stated above in \ref{theorem:geomformulation}, this is indeed the case, so long as $M$ is finite type.

\begin{theorem}\label{theorem:repisdifop}
Let $M$ be a finite $A \cV$-module. Then the associated representation $\rho: \cV \rightarrow \mathrm{At}(M)$ is a differential operator of order less or equal than $\rank (M)^2 $. 
\end{theorem}
\begin{proof}
    Set $p=\rank(M)^2+1$. By Corollary \ref{corollary:omegaisdifferenrial}, 
    \[
\Omega((f_1,\dots, f_p), \eta) m=\sum_{H\subset \{ 1,\dots, p\}} (-1)^{|H|} \left (\prod_{i\in H}f_i \right) \rho\left (\left ( \prod_{i\in H}f_i \right)\eta \right)=0
    \]
    for each $m\in M$, $\eta\in \cV$, $f_1,\dots, f_p\in A$. By \cite[Proposition 16.8.8]{Gro67}, $\rho$ is a differential operator.
\end{proof}
This fact does not extend to $A\cV$-modules that are not finitely generated as $A$-module. For instance, $A\#U(\cV) $ is naturally an $A\cV$-module, but the associated Lie representation of $\cV$ is not a differential operator. The previous proposition corroborate Proposition \ref{proposition:infeqsheafiseqvtoavmod}, in the sense that $\rho : \cV \rightarrow \mathrm{At} (\cM) $ being a differential operator implies that it sheafifies in the Zariski (and indeed \'{e}tale topology), and the associated sheaf $\tilde{\rho}: \tilde{\cV} \rightarrow \mathrm{At}(\tilde{M})$ is also a differential operator.

\section*{Acknowledgments}	
Both authors thank Colin Ingalls, Vyacheslav Futorny and in particular Yuly Billig for numerous helpful discussions.
H.\,R.\,was financed by São Paulo Research Foundation (FAPESP) (grant 2020/13811-0, 2022/00184-3).


\begin{thebibliography}{}
\bibitem{BS88} A. A. Be\u{\i}linson\ and\ V. V. Schechtman, Determinant bundles and Virasoro algebras, Comm. Math. Phys. {\bf 118} (1988), no.~4, 651--701.
\bibitem{BF16} Y. Billig\ and\ V. Futorny, Classification of irreducible representations of Lie algebra of vector fields on a torus, J. Reine Angew. Math. {\bf 720} (2016), 199--216.
\bibitem{BF18} Y. Billig\ and\ V. Futorny, Lie algebras of vector fields on smooth affine varieties, Comm. Algebra {\bf 46} (2018), no.~8, 3413--3429.
\bibitem{BFN19} Y. Billig, V. Futorny\ and\ J. Nilsson, Representations of Lie algebras of vector fields on affine varieties, Israel J. Math. {\bf 233} (2019), no.~1, 379--399.
\bibitem{Car09} E. Cartan, Les groupes de transformations continus, infinis, simples, Ann. Sci. \'{E}cole Norm. Sup. (3) {\bf 26} (1909), 93--161.
\bibitem{Eis95} D. Eisenbud, {\it Commutative algebra}, Graduate Texts in Mathematics, 150, Springer-Verlag, New York, 1995.
\bibitem{GS22} D. Grantcharov\ and\ V. Serganova, Simple weight modules with finite weight multiplicities over the Lie algebra of polynomial vector fields, Journal für die reine und angewandte Mathematik (Crelles Journal), 2022. https://doi.org/10.1515/crelle-2022-0053
\bibitem{Gro67} A. Grothendieck, \'{E}l\'{e}ments de g\'{e}om\'{e}trie alg\'{e}brique. IV. \'{E}tude locale des sch\'{e}mas et des morphismes de sch\'{e}mas IV, Inst. Hautes \'{E}tudes Sci. Publ. Math. No. 32 (1967), 361 pp.
\bibitem{Jor86} D. A. Jordan, On the ideals of a Lie algebra of derivations, J. London Math. Soc. (2) {\bf 33} (1986), no.~1, 33--39
\bibitem{Lie80} S. Lie, Theorie der Transformationsgruppen I, Math. Ann. {\bf 16} (1880), no.~4, 441--528.
\bibitem{Sie96} T. Siebert, Lie algebras of derivations and affine algebraic geometry over fields of characteristic $0$, Math. Ann. {\bf 305} (1996), no.~2, 271--286. 

\end{thebibliography}

\end{document}